\documentclass[12pt,reqno]{amsart}
\usepackage{amsmath}
\usepackage{amsthm}
\usepackage{amssymb}
  \usepackage{bm}
  \usepackage{geometry}
  \usepackage[mathscr]{euscript}
  \usepackage{hyperref}
  \usepackage{url}
\newcommand{\bburl}[1]{\textcolor{blue}{\url{#1}}}
  \usepackage{paralist}
  \usepackage{tikz}
	\usepackage{amsthm}
    \theoremstyle{plain}
\usepackage{amsfonts}
\usepackage{geometry}
\usepackage{enumitem}
\usepackage {amssymb}
\usepackage{amsthm}
\usepackage{fullpage}
\usepackage{gensymb}

\usepackage{mathrsfs} 

\usepackage{verbatim}

\numberwithin{equation}{section}
\numberwithin{part}{section}

\newtheorem{lemma}[part]{Lemma}
\newtheorem{question}[part]{Question}

\newtheorem{thm}[part]{Theorem}

\newtheorem{prop}[part]{Proposition}

\newcommand{\Z}{{\mathbb Z}}

\newcommand{\R}{{\mathbb R}}

\newcommand{\abs}[1]{\lvert#1\rvert}
\newcommand{\qw}[1]{\left(#1\right)}
\newcommand{\mc}[1]{\mathcal{#1}}
\newcommand{\norm}[1]{\left\|#1\right\|}
\newcommand{\pez}[1]{\left(#1\right)}

\DeclareMathAlphabet{\curly}{U}{rsfs}{m}{n}  
 
\newcommand{\bx}{\mathbf{x}}
\newcommand{\by}{\mathbf{y}}

\title{Dimensional lower bounds for Falconer type incidence theorems}
\author{Jonathan De{W}itt, Kevin Ford, Eli Goldstein, Steven J. Miller, Gwyneth Moreland, Eyvindur A. Palsson, Steven Senger}
\address{Department of Mathematics \& Statistics, Haverford College, Haverford, PA}
\email{\textcolor{blue}{jdewitt@haverford.edu}}
\address{Department of Mathematics, University of Illinois at Urbana-Champaign, Urbana, IL}
\email{\textcolor{blue}{ford@math.uiuc.edu}}
\address{Department of Mathematics \& Statistics, Williams College, Williamstown, MA}
\email{\textcolor{blue}{esg2@williams.edu}}
\address{Department of Mathematics \& Statistics, Williams College, Williamstown, MA}
\email{\textcolor{blue}{sjm1@williams.edu}, \textcolor{blue}{Steven.Miller.MC.96@aya.yale.edu}}
\address{Department of Mathematics, University of Michigan, Ann Arbor, MI}
\email{\textcolor{blue}{gwynm@umich.edu}}
\address{Department of Mathematics \& Statistics, Williams College, Williamstown, MA}
\curraddr{Department of Mathematics, Virginia Tech, Blacksburg, VA}
\email{\textcolor{blue}{eap2@williams.edu}, \textcolor{blue}{palsson@vt.edu}}
\address{Department of Mathematics, Missouri State University, Springfield, MO}
\email{\textcolor{blue}{StevenSenger@MissouriState.edu}}
\date{\today}
\thanks{The first, third, fourth and fifth authors were supported in part by National Science Foundation grants DMS1265673, DMS1561945 and DMS1347804. The second listed author was supported in part by National Science Foundation grant DMS1501982 and the sixth listed author was supported supported in part by Simons Foundation Grant \#360560.}

\begin{document}
\begin{abstract}
Let $1 \leq k \leq d$ and consider a subset $E\subset \mathbb{R}^d$. In this paper, we study the problem of how large the Hausdorff dimension of $E$ must be in order for the set of distinct noncongruent $k$-simplices in $E$ (that is, noncongruent point configurations of $k+1$ points from $E$) to have positive Lebesgue measure. This generalizes the $k=1$ case, the well-known Falconer distance problem and a major open problem in geometric measure theory. Many results on Falconer type theorems have been established through incidence theorems, which generally establish sufficient but not necessary conditions for the point configuration theorems. We establish a dimensional lower threshold of $\frac{d+1}{2}$ on incidence theorems for $k$-simplices where $k\leq d \leq 2k+1$ by generalizing an example of Mattila. We also prove a dimensional lower threshold of $\frac{d+1}{2}$ on incidence theorems for triangles in a convex setting in every dimension greater than $3$. This last result generalizes work by Iosevich and Senger on distances that was built on a construction by Valtr. The final result utilizes number-theoretic machinery to estimate the number of solutions to a Diophantine equation.
\end{abstract}

\maketitle

\section{Introduction}

The Falconer distance problem, introduced in \cite{F85}, can be stated as follows: How large does the Hausdorff dimension of $E\subset\mathbb{R}^d$ need to be to ensure that the Euclidean distance set $\Delta(E)=\{\, |x-y|\, : x,y \in E\}\subset\mathbb R$ has positive one-dimensional Lebesgue measure? This problem can be viewed as a continuous analogue of the famous Erd\H{o}s distinct distance problem \cite{GIS,M95}. The current best  partial results, due to Wolff \cite{W99} in the plane and Erdo\u{g}an \cite{E05} in higher dimensions, say that the one-dimensional Lebesgue measure of $\Delta(E)$, denoted ${\mathcal L}^1(\Delta(E))$, is indeed positive if $\dim_{{\mathcal H}}(E)>\frac{d}{2}+\frac{1}{3}$ where $\dim_{{\mathcal H}}(E)$ denotes the Hausdorff dimension of $E$. As distance is a configuration that only involves two points, analogous questions can be posed for configurations that involve more points.  For example, we may consider the set of noncongruent triples of points in $E$, that is, points which form noncongruent triangles. In the discrete setting such questions have been studied for decades \cite{PS}, while recently there has been a flurry of activity in the continuous setting where angles \cite{IMP}, simplices \cite{EIH,GI,GILP15, GILP16}, volumes \cite{GIM}, and a more general approach to multi-point configurations \cite{GGIP} have been examined.

In this paper the point configurations we focus on are simplices. For $d \ge 2$ and $1 \le k \le d$ we say, following \cite{GILP15}, that the \textit{set of distinct noncongruent $k$-simplices determined by $E \subseteq \R^d$} is $T_{k,d}(E) := E^{k+1}{/\sim}$, where $(\bx_1,\dots,\bx_{k+1}) \sim (\by_1, \dots, \by_{k+1})$ provided that $(\bx_1,\dots,\bx_{k+1})$, $(\by_{1},\dots,\by_{k+1})$ form non-degenerate $k$-simplices and $|\bx_i - \bx_j| = |\by_i - \by_j|$ for all $1 \le i < j \le k+1$. We can map $T_{k,d}(E) \hookrightarrow \R^{\binom{k+1}{2}}$ by mapping a $k$-simplex to the $\binom{k+1}{2}$-tuple of its distances, thus it makes sense to take the $\binom{k+1}{2}$-dimensional Lebesgue measure of $T_{k,d}(E)$. Define $\alpha_{k,d}$ to be the infimum of all $\alpha$ for which $\text{dim}_{\mathcal{H}}(E) > \alpha$ implies $\mathcal{L}^{\binom{k+1}{2}}(T_{k,d}(E)) > 0$. The first Falconer type theorem for simplices was established by Greenleaf and Iosevich \cite{GI}, where, in the special case of triangles in the plane, they established the upper bound $\alpha_{2,2}\leq \frac{7}{4}$. This result was extended by Grafakos, Greenleaf, Iosevich and the sixth listed author \cite{GGIP} to the following upper bound for all simplices in all dimensions: $\alpha_{k,d}\leq d-\frac{d-1}{2k}$. These results were further improved by Greenleaf, Iosevich, Lu and the sixth listed author \cite{GILP15} to $\alpha_{2,2}\leq \frac{8}{5}$ and in the general case $\alpha_{k,d}\leq \frac{dk+1}{k+1}$. Using a different approach Erdo\u{g}an, Iosevich and Hart \cite{EIH} obtained the upper bound $\alpha_{k,d}\leq \frac{d+k+1}{2}$, which recently has been improved by Greenleaf, Iosevich, Lu and the sixth listed author \cite{GILP16} to $\alpha_{k,d}\leq \frac{d+k}{2}$. In certain situations these bounds beat the previous ones, but in some of the most natural situations, such as when $d=k$, they only give trivial information.

All these positive results naturally lead to the question of whether they are sharp. Using a set obtained by a suitable scaling of the thickened integer lattice, Falconer \cite{F86} showed a lower bound of $\frac{d}{2}$ for his distance problem, i.e., $\frac{d}{2}\leq \alpha_{1,d}$. This led him to conjecture $\alpha_{1,d}=\frac{d}{2}$, which remains open. For higher order simplices the trivial observation that $\max\left\{k-1,\frac{d}{2}\right\}\leq \alpha_{k,d}$ had been made. The first part of the lower bound says that there are not many $k$-simplices in a $k-1$-dimensional set, e.g. there are not many triangles on a line. The second part of the lower bound says that if there are many different $k$-simplices then there are many different distances, therefore Falconer's lower bound for distances also applies for $k$-simplices. The only non-trivial lower bound for higher order simplices is one for triangles in the plane, obtained by Erdo\u{g}an and Iosevich, but first published in \cite{GILP15}, that says $\frac{3}{2}\leq \alpha_{2,2}$. This was obtained by counting triangles in an integer lattice. In summary, not much is known about sharpness of results but in this paper we address some questions about the sharpness of the techniques used.

\subsection{Lower bounds for Falconer type incidence theorems}
In his original paper Falconer obtained the threshold $\alpha_{1,d}\ge\frac{d}{2}+\frac{1}{2}$ by proving an incidence theorem. He showed that if the Hausdorff dimension of $E$ is above $\frac{d}{2}+\frac{1}{2}$ then uniform estimates for $t>0$ of the form
\begin{equation}\label{Falconer}
\nu \times \nu \{(\bx,\by): t - \epsilon \ \le \  |\bx-\by|\ \le \  t + \epsilon \}\ \ll_{d,E}\ \epsilon\end{equation}
hold for a Frostman measure $\nu$ supported on $E$. A Frostman measure is a probability measure, so one can interpret the above inequality as the probability that $|\bx-\by|$ is near a fixed distance $t$. In \cite{GI} and \cite{GGIP}, a similar approach was taken and incidence theorems of the type
\begin{equation}\label{estinc}
\nu \times \dots \times \nu \{(\bx_1, \dots, \bx_{k+1}): t_{ij}-\epsilon \ \le \  |\bx_i-\bx_j| \ \le \  t_{ij}+\epsilon \;\; (i\ne j) \}\ \ll_{k,d,E}\ \epsilon^{\binom{k+1}{2}},
\end{equation}
where ${\{t_{ij} \}}_{1 \leq i<j \leq k+1}$ is a collection of positive real numbers, were proven.

In \cite{M85}, Mattila showed for $d=2$ that Falconer's incidence theorem \eqref{Falconer} does not in general hold if the Hausdorff dimension of $E$ is strictly less than $\frac{d}{2}+\frac{1}{2}=\frac32$. This means that his original approach is sharp in terms of the technique used. Note that this does not imply that his distance theorem is sharp, and as mentioned before, it has since been improved from $\frac{d}{2}+\frac{1}{2}$ down to $\frac{d}{2}+\frac{1}{3}$. In \cite{GI}, Greenleaf and Iosevich extended Mattila's example to triangles in the plane ($k=2$,$d=2$) and showed that the incidence theorem they obtained does not in general hold if the Hausdorff dimension of $E$ is strictly less than $\frac{7}{4}$, which shows that their incidence theorem is sharp. Again this does not imply that the point configuration problem is sharp and indeed the dimensional threshold was improved to $\frac{8}{5}$ in \cite{GILP15}. This brings us to the first result of the paper.

\begin{prop}\label{prop:lowerboundincidencethm}
For any $k$ and $d$ where $k\le d\le 2k+1$,
the incidence estimate for $k$-simplices, i.e., the estimate given in \eqref{estinc}, can fail for measures supported on sets with Hausdorff dimension less than $\frac{d+1}{2}$.
\end{prop}

We prove this proposition in Section \ref{sec:prooflowerboundincidencethm} by extending the constructions of Mattila and Greenleaf and Iosevich. We remark that the dimensional threshold obtained in the incidence theorems in \cite{GGIP} is $d-\frac{d-1}{2k}$, so for $k>1$, unlike the previously mentioned results, there is a gap between the threshold from the construction and the threshold from the incidence theorems.

\subsection{Lower bounds for Falconer type incidence theorems in convex domains}
The key ingredient in both Falconer's incidence theorem and the incidence theorem from \cite{GGIP} is that if $\sigma$ denotes the Lebesgue measure on the unit sphere, then
\begin{equation*} |\widehat{\sigma}(\xi)|\ \ll |\xi|^{-\frac{d-1}{2}}.\end{equation*}
This implies that both incidence theorems still hold if the Euclidean distance $|\cdot|$ is replaced by $\|\cdot \|_{B}$, where $B$ is a symmetric convex body with a smooth boundary and everywhere non-vanishing Gaussian curvature.

Mattila's construction, which shows his incidence theorem does not in general hold if the Hausdorff dimension of $E$ is strictly less than $\frac{d}{2}+\frac{1}{2}$, was originally proven in the case $d=2$ and extends to $d=3$ but does not seem to extend to higher dimensions. Iosevich and Senger \cite{IS} showed, building on a construction by Valtr \cite{V}, that the more general incidence theorem involving a norm derived from a symmetric convex body can fail if the Hausdorff dimension of $E$ is strictly less than $\frac{d}{2}+\frac{1}{2}$ for all $d\geq 2$. Our second and main result is the following theorem that establishes an analogous result in the case of triangles ($k=2$).

\begin{thm}\label{triangles}
For $d>3$ there exists a symmetric convex body $B$ with a smooth boundary and non-vanishing Gaussian curvature such that for any $s<(d+1)/2$, there exists a Borel measure $\mu_s$ such that $I_s(\mu_s) =O(1)$ 
 and
\begin{equation}\label{incidence_mu3}
\limsup_{\epsilon\rightarrow 0} \epsilon^{-3}\mu_s\times\mu_s\times \mu_s\{(\bx_1,\bx_2,\bx_3): 1-\epsilon \ \le \ \| \bx_i-\bx_j\|_B\ \le \ 1+\epsilon \;\; (i<j)\}\ = \ \infty;
\end{equation}
i.e., the incidence theorem fails.
\end{thm}

Here $I_s(\mu_s)$ denotes the energy integral
\begin{equation*}
I_s(\mu_s)\ = \ \int\int \abs{\bx-\by}^{-s}\,d\mu_s(\bx)\,d\mu_s (\by),
\end{equation*}
and the condition $I_s(\mu_s)=O(1)$ simply means that the measure $\mu_s$ is supported on a set of Hausdorff dimension at least $s$.

We prove this proposition in Section \ref{sec:countingtrianglesinconvexnorm}. The main ingredient in the proof is some interesting number theory that arises when we count the number of equilateral triangles in this convex norm. We attempted to extend this result to tetrahedra but encounter a harder number theory problem. In Section \ref{sec:countingtetrahedrainconvexnorm} we set up the problem and pose the number theory problem that can resolve it.

\subsection{Notation}

Vector quantities will be denoted in boldface type, e.g. $\bx$ or $\by_j$.  
The notation $f=O(g)$, $f\ll g$, $g \gg f$ and $g=\Omega(f)$ have the usual meaning, 
that there is a positive constant $C$ so
that $|f| \le C |g|$ throughout the domain of $f$.  If the constant $C$ depends on any parameter,
then this is indicated by a subscript, e.g. $f(x) = O_\epsilon (x^{1+\epsilon})$.  The notation
$f \asymp g$ means that both $f\ll g$ and $g\ll f$ hold, that is, there are positive constants
$c_1,c_2$ such that $c_1 g \le f \le c_2 g$ (we can say that $f$ and $g$ have the same order).


\section{Proof of Proposition \ref{prop:lowerboundincidencethm}}\label{sec:prooflowerboundincidencethm}

We proceed by generalizing the example of Mattila presented in \cite{IS}, which was introduced in \cite{M85}, and generalized to triangles in the plane in \cite{GI}. For $0\le \alpha\le 1$, let $C_{\alpha}$ denote the standard $\alpha$-dimensional Cantor set contained in the interval $[0,1]$. Set $F_{\alpha_i}= C_{\alpha_i} \cup \qw{C_{\alpha_i}-1}$, and let $E_d=F_{\alpha_1}\times\cdots \times F_{\alpha_d}$, where we give $E$ the product measure arising from the the $\alpha$-dimensional Hausdorff measure, $\mc{H}^\alpha$, on $F_\alpha$. Hence $\dim_{\mc{H}} E_d=\sum_{i=1}^{d} \alpha_i$ \cite{M95}.

Now fix a point $x$ in $E$, and pick out a $k$-simplex of $E$ containing $x$ so that each of the edges from $x$ of the simplex have length $1$ and they are all orthogonal at $x$. Then we may fatten each of the nodes of the simplex besides $x$ to an $\epsilon\times\cdots\times \epsilon\times \sqrt{\epsilon}\times\cdots \times \sqrt{\epsilon}$ box, where there are $k$ sides of length $\epsilon$ and $d-k$ sides of length $\sqrt{\epsilon}$. Each of the points within these boxes form a $k+1$ simplex along with $x$. Further, each axis aligned box has measure
\begin{equation*}
\epsilon^{\sum_{i=1}^k \alpha_i +\sum_{i=k+1}^d \alpha_i/2}.
\end{equation*}
Thus the combined measure of each of the $k$ boxes we have selected is
\begin{equation}\label{llll}
\epsilon^{k\qw{\sum_{i=1}^k \alpha_i +\sum_{i=k+1}^d \alpha_i/2}}.
\end{equation}
Integrating over all possible values of $x$, we see that \eqref{llll} is a lower bound for the left hand side of \eqref{estinc}. In order for this bound to be a larger order of magnitude than $\epsilon^{\binom{k+1}{2}}$, we must have
\begin{equation*}
\sum_{i=1}^k \alpha_i +\sum_{i=k+1}^d \alpha_i/2\ \le \ \frac{k+1}{2},
\end{equation*}
while simultaneously satisfying $\dim_{\mc{H}} E \le \frac{d+1}{2}$ and $0\le \alpha_i\le 1$. Hence for $1\le i\le k$ set $\alpha_i$ equal to $(2k+1-d)/2k$ and for $k+1\le i\le d$ set $\alpha_i$ equal to $1$. Then $\sum_{i=1}^d \alpha_i=\frac{d+1}{2}$ while the previous sum is equal to $\frac{k+1}{2}$. Note that the choice of the first $k$ values of $\alpha_i$ gives us the restriction that $d\leq 2k+1.$



\section{Counting Triangles in a Convex Norm}\label{sec:countingtrianglesinconvexnorm}

To begin, we define a convex body, $B$, which induces a norm on $\R^d$. Let
\begin{align}
\label{BU}
B_U&\ = \ \left\{(x_1,x_2,\dots,x_d)\in \R^d: \sum_{i=1}^{d-1} x_i^2\ \le \ 1  \text{ and } x_d\ = \ 1-(x_1^2+\cdots+x_{d-1}^2)\right\}\\
\label{BL}
B_L&\ = \ \left\{(x_1,x_2,\dots,x_d)\in \R^d: \sum_{i=1}^{d-1} x_i^2\ \le \ 1  \text{ and } x_d\ = \ -1+(x_1^2+\cdots+x_{d-1}^2)\right\}.
\end{align}

Then $B_L=-B_U$, and define $B=B_U\cup B_L$ and the induced norm $\norm{\cdot}_B$ with unit ball $B$. In other words, the point $(x_1,...,x_d)$ in $\R^d$ is at unit distance from the origin if $\sum_{i=1}^{d-1} x_i^2\le 1$ and either $x_{d}=1-\sum_{i=1}^{d-1} x_i^2$ or $x_{d}=-1+\sum_{i=1}^{d-1} x_i^2$, depending on whether $\bx$ lies on the upper or lower hemisphere of the unit paraboloid, respectively. Hence the points $\mathbf{x}$ and $\mathbf{y}$ are at unit distance if $\mathbf{x}-\mathbf{y}$ lies on the unit paraboloid.

\newcommand{\LL}{\curly{L}}
Now consider the lattice
\begin{equation*}
\LL_n = \ \left\{\pez{\frac{i_1}{n},\cdots ,\frac{i_{d-1}}{n},\frac{i_d}{n^2}}: (i_1,\dots,i_{d})\in \Z \right\}.
\end{equation*}

\begin{lemma}\label{lem:valtrtri}The number of unit equilateral triangles $(0,\bx,\by)$ with $\bx,\by\in B\cap \LL_n$ is $\Omega(n^{2d-4})$ for $d>3$.
\end{lemma}

\begin{proof}
It suffices to examine choices of points $\mathbf{x}$ on $B_L$ and $\mathbf{y}$ on $B_U$.  These points are of the form
\begin{align*}
\mathbf{x} &\ = \    \left( \frac{x_1}{n}, \ldots, \frac{x_{d-1}}{n},-1 + \sum_{i=1}^{d-1}\frac{x_i^2}{n^2}\right)\nonumber\\
\mathbf{y} &\ = \   \left( \frac{y_1}{n}, \ldots, \frac{y_{d-1}}{n},\,\,\,\,\,1 - \sum_{i=1}^{d-1}\frac{y_i^2}{n^2}\right)
\end{align*}
and are depicted in Figure \ref{fig:balls}. Now we must ensure that $\mathbf{x}-\mathbf{y}$ is also of unit length, hence we must have that
\begin{equation}\label{eq:req}
\sum_{i=1}^{d-1} (x_i-y_i)^2\le n^2,
\end{equation}
and $\mathbf{x}-\mathbf{y}$ lies on either $B_{U}$ or $B_{L}$. 

\begin{equation*}\mathbf{x}-\mathbf{y} \ = \  \left( \frac{x_1 - y_1}{n}, \ldots, \frac{x_{d-1} - y_{d-1}}{n}, -2 + \sum_{i=1}^{d-1}\frac{x_i^2 + y_i^2}{n^2}\right),\end{equation*}
and so for this vector to lie on $B_L$, the final component of this vector must satisfy:
\begin{equation*}-2 + \sum_{i=1}^{d-1}\frac{x_i^2 + y_i^2}{n^2}\ =\ -1 + \sum_{i=1}^{d-1}\frac{(x_i - y_i)^2}{n^2}.\end{equation*}
This yields a Diophantine equation whose number of solutions we will bound below:
\begin{equation}\label{eq:goal}
\sum_{i=1}^{d-1}x_iy_i\ =\ \frac{n^2}{2},
\end{equation}
where the $x_i,y_i$ are integers satisfying $\sum_{i=1}^{d-1}x_i^2\le n^2$, $\sum_{i=1}^{d-1}y_i^2\le n^2$, and $\sum_{i=1}^{d-1} (x_i-y_i)^2\le n^2$.

\begin{figure}
\begin{center}
\includegraphics{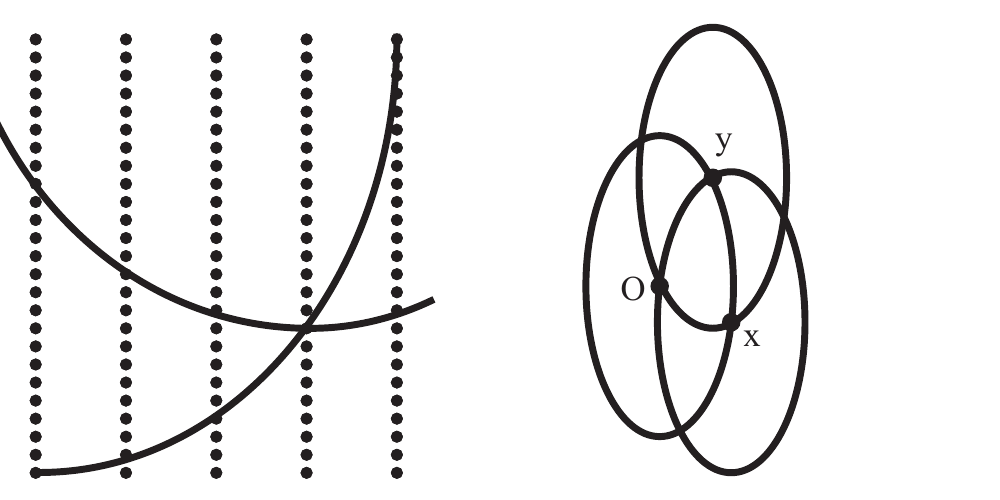}
\caption{Illustration of the configuration of $\mathbf{x}$ and $\mathbf{y}$ in this convex norm.}\label{fig:balls}
\end{center}
\end{figure}

We now apply the following.

\begin{lemma}\label{lem:estimate}
Fix $\lambda > 1$.  Let $a_{n}$ be the number of solutions to
\begin{equation*}
n \ =\ xi+yj,
\end{equation*}
where each of $i,j,x,y\in [-\lambda\sqrt{n},\lambda\sqrt{n}]\cap \Z$.  Then $a_n\geq c_\lambda n+O\left(\sqrt n \log n\right)$, where
$$c_\lambda =\frac{9\lambda^2-(6+2\pi)\lambda + (2\pi-3)}{4\pi^2}.$$
\end{lemma}
\begin{proof}
To prove this, we will use the following lemma, which is Theorem 330 from Hardy and Wright, \cite{HW}.

\begin{lemma}\label{330}
Let $\phi(m)$ denote the number of positive integers less than $n$ that are coprime to $n$. Then $$\Phi(m) = \sum_{j=1}^m \phi(m) = \frac{3m^2}{\pi^2} + O(m \log m).$$
\end{lemma}

Consider the set $A = [\sqrt{n}, \lambda\sqrt{n}]$. We want a lower bound on the number of pairs of positive integers, $i,j \in A$, that are relatively prime. Set $m_1 =\left\lceil\left(\frac{\lambda+1}{2}\right) \sqrt n\right\rceil$ and $m_2 =\left\lfloor\lambda \sqrt n\right\rfloor$. Given an integer, $m\in A$, the we can estimate the number of integers in $A$ that are relatively prime to it from below by $\phi(m) - \sqrt n.$ Summing this over the $\frac{1}{2}(\lambda -1)\sqrt n$ terms in the upper half of $A$ gives us a lower bound on the number of coprime pairs of integers in $A$: 
$$\Phi(m_2)-\Phi(m_1)-\left(\frac{1}{2}(\lambda-1)\sqrt n\right)\sqrt n.$$
We apply Lemma \ref{330} twice to get that this is bounded below by
$$\frac{3\left(\lambda^2-\frac{\lambda^2+2\lambda+1}{4}\right)n}{\pi^2}-\frac{(\lambda-1)n}{2}+O(\sqrt n \log n).$$
Putting this together, we see that for sufficiently large $n$ there will be at least
$$\left(\frac{9\lambda^2-(6+2\pi)\lambda + (2\pi-3)}{4\pi^2}\right)n + O\left(\sqrt n \log n\right)= c_\lambda n + O\left(\sqrt n \log n\right)$$
pairs of coprime numbers in $A$. We now show that each such pair gives rise to a solution. Namely, there is a solution of $ix+jy=n$ with $0<x\leq j$. We have $y<n/j \leq \sqrt{n}$ and
\begin{equation*}
 y \geq \frac{n-ij}{j} =  \frac{n}{j}-i \geq -\lambda \sqrt{n}.\qedhere
\end{equation*}
\end{proof}

We can now show that there are $\Omega(n^{2d-4})$ solutions to the system \eqref{eq:goal}.
 We pick a range $[\sqrt{a} n,\sqrt{b}n]$ for the $x_j$ and $y_j$ to vary in where $1\le j\le d-3$, so that to each value they assume we may apply the previous lemma and obtain on the order of $n^2$ solutions in the remaining variables $x_{d-2},x_{d-1},y_{d-2},y_{d-1}$.  We put $a=\frac{1}{4(d-3)}$ and $b=\frac{1}{3(d-3)}$ and find that there are $\alpha n^{2d-6}$ choices of the $x_j$ and $y_j$, with $j=1, \dots, d-3$, with
$$\alpha = \left(\sqrt{\frac{1}{3(d-3)}}-\sqrt{\frac{1}{4(d-3)}}\right)^{2d-6},$$
that satisfy
\begin{equation}\label{eq:s1}
\frac{n^2}{6}\ \le\ \frac{n^2}{2}-\sum_{i=1}^{d-3} x_jy_j\ \le\ \frac{n^2}{4},
\end{equation}
and
\begin{equation}\label{eq:s2}
\sum_{j=1}^{d-3} x_j^2\ \le\ \frac{n^2}{3},\quad
\sum_{j=1}^{d-3} y_j^2\ \le\ \frac{n^2}{3},\quad
\sum_{j=1}^{d-3} (x_j-y_j)^2\ \le\ \frac{n^2}{6}.
\end{equation}

Hence for a given choice of $x_j,y_j\in [\sqrt{a} n,\sqrt{b}n]$, we are left solving an equation
\begin{equation*}
m\ = \ x_{d-2}y_{d-2}+x_{d-1}y_{d-1},
\end{equation*}
where $n^2/6\le m\le n^2/4$.  Let $\lambda=1.05$.  By Lemma \ref{lem:estimate} for each $m$ in this range there are at least $c_\lambda n^2+ O\left( n \log n\right)$ (with $c_\lambda = c_{1.05} \approx .308\dots)$ solutions $x_{d-2},y_{d-2},x_{d-1},y_{d-1},$ with the absolute value of each of these numbers being less than $\lambda n / 2$.  For each such solution, we have
$$\sum_{j=1}^{d-1} x_j^2\le (1/3+\lambda^2/2)n^2\le n^2,$$
and similarly for $\sum_{j=1}^{d-1} y_j^2$. Further, by using the lower bound on $m$ to estimate the cross term and the bounds on the absolute values of $x_{d-2},y_{d-2},x_{d-1},$ and $y_{d-1},$ we obtain that
\begin{align*}
\sum_{j=1}^{d-1} (x_j-y_j)^2 &\leq \left(\sum_{j=1}^{d-3} (x_j-y_j)^2\right) + \left(x_{d-2}^2-2x_{d-2}y_{d-2} + y_{d-2}^2\right)+\left(x_{d-1}^2-2x_{d-1}y_{d-1} + y_{d-1}^2\right)\\
&\leq \frac{n^2}{6} + x_{d-2}^2+y_{d-2}^2+x_{d-1}^2+y_{d-1}^2 -2m \leq \left(\frac{1}{6}+\lambda^2-\frac{1}{3}\right)n^2\le .94n^2<n^2,
\end{align*}
as required by \eqref{eq:req}.

Hence there are at least $c_\lambda n^2+O\left(n \log n\right)$ choices of $x_{d-2},x_{d-1},y_{d-2},y_{d-1}$ that allow us to solve the equation with the three additional constraints. Thus, as there are $\alpha n^{2d-6}$ choices of $x_{j},y_{j}\in [\sqrt{a} n,\sqrt{b}n]$ for $1\le j\le d-3$, satisfying \eqref{eq:s1} and \eqref{eq:s2}, there are $c_\lambda \alpha n^{2d-4}+O\left(n \log n\right)=\Omega\left(n^{2d-4}\right)$ solutions to the equation \eqref{eq:goal}.
\end{proof}

\begin{proof}[Proof of Theorem \ref{triangles}]
We first count the total number of unit triangles in $\LL_n$ that have a point in $[0,1)^d$. 
For each one of the $n^{d+1}$ points $P$ in $\LL_n\cap [0,1)^d$, the translation of $\LL_n$
by $P$ is just $\LL_n$ itself, hence by  Lemma \ref{lem:valtrtri} there are $\Omega(n^{2d-4})$ distinct unit triangles in $\LL_n$ with one point being $P$.
 So, we have  $\Omega(n^{3d-3})$ total unit triangles in $\LL_n$ which include a point in $[0,1)^d$.

We now construct the measure $\mu_s$---this construction is analogous to that in \cite{IS}, except that we need to consider a larger configuration of points. To do this, partition space into lattice cubes of side length $\epsilon^{s}=1/n^{d+1}$ for some large integer $n$ and $\frac{d}{2}\le s\le \frac{d+1}{2}$ as this range is non-trivial.  Now set $\mu_s$ to be the Lebesgue measure on those cubes containing a point of $\LL_n$, our lattice, normalized by $\epsilon^{s-d}$ and furthermore restricted to lie in some large box, i.e.,
\begin{equation*}
d\mu_s(x)\ = \ \epsilon^{s-d}\sum_{p\in \LL_n\cap[-2,2]^d} \chi_{R_\epsilon(p)}(x)\,dx,
\end{equation*}
where $R_{\epsilon}(p)$ denotes the cube of side-length $\epsilon$ centered at $p$. It then follows by Lemma 2.1 in \cite{IS} that
\begin{equation*}
I_s(\mu_s) =O(1),
\end{equation*}
and moreover that by normalizing we can take $\int \,d\mu_s=1$ as well.

We now have that
\begin{equation*}
\mu_s\times\mu_s\times \mu_s\{(\bx_1,\bx_2,\bx_3): 1 \le \| \bx_i-\bx_j\|_B \le 1+\epsilon \; (i<j)\} \gg \epsilon^{3s}n^{3(d+1)-6}= \epsilon^{s\frac{6}{d+1}},
\end{equation*}
as each triangle we counted in our point configuration contributes $\epsilon^{3s}$ to the measure.  Therefore, the estimate \eqref{incidence_mu3} fails for all $s< \frac{d+1}{2}$,
for every $d>3$.
\end{proof}


\section{Counting Tetrahedra in a Convex Norm}\label{sec:countingtetrahedrainconvexnorm}

In this section we give the natural generalization of the previous section's argument to counting tetrahedra in the Valtr construction. We present a system of equations which governs the number of tetrahedra and present this number-theoretic problem as an open question.

Define $B_U$ and $B_L$ as above in (\ref{BU}), (\ref{BL}).  We now consider the unit ball centered about the origin, and  we choose points $\mathbf{x}$ and $\mathbf{y}$ on $B_U$, and $\mathbf{z}$ on $B_L$.

Recall that these points are of the form
\begin{align*}
\mathbf{x} &\ = \    \left( \frac{x_1}{n},\ \ldots,\ \frac{x_{d-1}}{n},\,\,\,\,1 - \sum_{i=1}^{d-1}\frac{x_i^2}{n^2}\right)\nonumber\\
\mathbf{y} &\ = \   \left( \frac{y_1}{n},\ \ldots,\ \frac{y_{d-1}}{n},\,\,\,\,\,1 - \sum_{i=1}^{d-1}\frac{y_i^2}{n^2}\right)\nonumber\\
\mathbf{z} &\ = \    \left( \frac{z_1}{n},\ \ldots,\ \frac{z_{d-1}}{n}, -1 + \sum_{i=1}^{d-1}\frac{z_i^2}{n^2}\right).
\end{align*}

We must ensure that the relevant vectors ($\mathbf{x}-\mathbf{y}$),  ($\mathbf{z}-\bx$), and  ($\mathbf{z}-\by$) are all of unit length.  For example, we may insist that
\begin{equation*} 
\mathbf{x}-\mathbf{y} \ = \  \left( \frac{x_1 - y_1}{n},\ \ldots,\ \frac{x_{d-1} - y_{d-1}}{n},\ \sum_{i=1}^{d-1}\frac{y_i^2 - x_i^2}{n^2}\right),\end{equation*}
lie on $B_U$, which implies that
\begin{equation*}\sum_{i=1}^{d-1}\frac{y_i^2 - x_i^2}{n^2}\ =\ 1 - \sum_{i=1}^{d-1}\frac{(y_i - x_i)^2}{n^2}.\end{equation*}
Simplifying the above equation and performing similar computations for the other cases ($\mathbf{z}-\bx$) and  ($\mathbf{z}-\by$) we obtain two additional equations, which are of a different form as $\mathbf{x},\mathbf{y}$ and $\mathbf{z}$ lie on opposite hemispheres. This yields our concluding question.

\question{ Let $a_n$ be the number of solutions to the following system of equations
\begin{align*}
&\sum_{i=1}^{d-1}y_i^2 \,\,\, \ =\ \frac{n^2}{2} + \sum_{i=1}^{d-1}x_iy_i\nonumber\\
&\sum_{i=1}^{d-1}x_iz_i \ =\ \frac{n^2}{2} \nonumber\\
&\sum_{i=1}^{d-1}y_iz_i \ =\  \frac{n^2}{2},
\end{align*}
where each $x_i, y_i, z_i \in [-n,n]$, and $\sum_{i=1}^{d-1} x_i^2\le n^2$, $\sum_{i=1}^{d-1} (x_i-y_i)^2\le n^2$, etc.  Is $a_n$ at least $\Omega(n^{3d - 6})$?
}

Given this solution, the expression of interest is
\begin{equation*}\epsilon^{4s}\cdot N^{4-\frac{9}{d+1}}\  \ll \epsilon^6.\end{equation*}


\end{document}